\documentclass{article}
\usepackage{geometry}
\usepackage[utf8]{inputenc}
\usepackage{amsmath}
\usepackage{amssymb}
\usepackage{amsthm}
\usepackage{mathtools}
\usepackage{braket}
\usepackage{graphicx}
\usepackage{todonotes}
\usepackage{hyperref}
\newcommand{\udots}{\mathinner{\mkern1mu\raise1pt\vbox{\kern7pt\hbox{.}}%
\mkern2mu\raise4pt\hbox{.}\raise7pt\hbox{.}\mkern1mu}}
\usepackage{sidecap}
\hypersetup{
	colorlinks=true,%
	linkcolor=black,%
	urlcolor=black,%
	citecolor=black
}

\theoremstyle{plain}
\newtheorem{proposition}{Proposition}
\newtheorem{lemma}[proposition]{Lemma}
\newtheorem{corollary}[proposition]{Corollary}
\newtheorem{theorem}[proposition]{Theorem}

\theoremstyle{definition}
\newtheorem{definition}[proposition]{Definition}
\newtheorem{example}[proposition]{Example}
\newtheorem{remark}[proposition]{Remark}

\DeclareMathOperator{\diag}{diag}
\newcommand{\C}{\mathbb{C}}
\newcommand{\Z}{\mathbb{Z}}
\newcommand{\Q}{\mathbb{Q}}
\newcommand{\R}{\mathbb{R}}
\newcommand{\K}{\mathbb{K}}
\newcommand{\X}{\mathcal{X}}
\newcommand{\Mat}{\mathrm{Mat}}
\newcommand{\GL}{\mathrm{GL}}
\newcommand{\AT}{\scriptsize\raisebox{\depth}{\rotatebox{180}{$T$}} }
\title{\textsc{X-matrices\\
}}
\author{Emanuele Borgonovo  \and Marco Artusa  \and Elmar Plischke \and Francesco
	Vigan\`o}
\date{}

\begin{document}
	
	\maketitle
	
	\noindent \textbf{Emanuele Borgonovo} \\
	Department of Decision Sciences and \\
	Bocconi Institute for Data Science and Analytics \\
	Via Roentgen 1, Milano, Italy \\
	\texttt{emanuele.borgonovo@unibocconi.it}
	\medskip
	
	\noindent \textbf{Marco Artusa} \\
	Dipartimento di Matematica \\
	Università degli Studi di Milano \\
	Via Saldini 50, Milano, Italy \\
	\texttt{m.artusa@clipnotes.it}
	\medskip
	
	\noindent \textbf{Elmar Plischke} \\
	Institut für Endlagerforschung \\
	Fachgebiet Endlagersysteme \\
	Technische Universität Clausthal \\
	Adolph-Roemer-Str.~2a, Clausthal-Zellerfeld, Germany \\
	\texttt{elmar.plischke@tu-clausthal.de}
	\medskip
	
	\noindent \textbf{Francesco Vigan\'o} \\
	Dipartimento di Matematica \\
	Università degli Studi di Milano \\
	Via Saldini 50, Milano, Italy \\
	\texttt{f.vigano@clipnotes.it}
	\medskip

	\begin{abstract}
		\noindent We evidence a family $\X$ of square matrices over a field $\K$, whose elements 
		will be called X-matrices. We show that this family is shape invariant under multiplication as well as transposition. We show that $\X$ is a (in general non-commutative) subring 
		of $\Mat(n,\K)$. Moreover, we analyse the condition for a matrix $A \in \X$ to be 
		invertible in $\X$. We also show that, if one adds a symmetry condition called here 
		bi-symmetry, then the set $\X^b$ of bi-symmetric X-matrices is a commutative subring of $\X$. We propose results for eigenvalue inclusion, showing that for X-matrices eigenvalues lie exactly on the boundary of Cassini ovals. It is shown that any monic polynomial on $ \R $ can be associated with a companion matrix in $ \X $.\\
		\textbf{Classification (AMS/MSC2010):} 15B99, 15A30, 15A21, 16S50.
	\end{abstract}
	

	\section{Introduction}
Diagonal matrices are shape-invariant under all main matrix 
	operations (sum, product, inversion, transposition). However, other popular families 
	of matrices do not show this invariance property. For instance, symmetric matrices are 
	closed with respect to sum and transposition, but the product of two
	symmetric matrices may be not symmetric; triangular matrices are not invariant under 
	transposition (upper triangular matrices become lower triangular). 
	Matrix sets such as $\mathrm{SL}(n,\K)$, $\mathrm{O}(n,\R)$ are subgroups of the 
	invertible matrices which are invariant under transposition; nonetheless, they are 
	not closed with respect to the sum. The purpose of this paper is to find an easy-to-handle 
	family of matrices which is both a sub-ring of $\Mat(n,\K)$ and shape-invariant under 
	multiplication and transposition as well. 
	We recall that shape invariance is important in numerical linear algebra, 
	when one wishes to keep the sparsity pattern intact \cite{ArbeGolu95}.  
	In this respect, we also note that while shape invariance under matrix multiplication 
	is frequently searched for numerical purposes, the family of matrices we are to discuss 
	is not only invariant under matrix multiplication but also under transposition and inversion.
		
	After a preliminary investigation of the notion of anti-transposition, we introduce the 
	family of X-matrices.  We prove that these matrices form a sub-ring of the square matrices 
	of size $n$ invariant under transposition. Regarding inversion, we focus on the inverse of 
	an X-matrix, when it exists: in fact, even if the family of X-matrices is a subring of the 
	ring of $n\times n$ matrices, 
	it is not obvious that the inverse of an X-matrix is still an X-matrix. However, we 
	show that this is indeed the case. In this result, it plays a favorable role the fact 
	the determinant of X-matrices can be expressed as a product of determinants of 
	inner $ 2 \times 2 $ quadratic submatrices.  
	
    The notion of anti-transposition forms a direct link to 
	skew-symmetric and anti-commutative matrices. These are pervading matrix theory 
	right from its beginning:
   ``It may be noticed in passing, that if $L$, $M$ are skew convertible matrices 
    of the order 2, and if these matrices are also such that 
    $L^2=-1$, $M^2=-1$, then putting $N=LM=-ML$, we obtain  
    \begin{alignat}{3}
    L^2&=-1, & M^2&=-1, &N^2&=-1,\\
    L&=MN=-NM, & M&=NL=-NL, & N&=LM=-ML  
    \end{alignat}
    which is a system of relations precisely similar to that in the theory of quaternions'' 
    \cite[p. 32]{Cayl89}. 

	We then discuss the commutativity of the product between X-matrices. We are able to find 
	a non-trivial sub-ring of the X-matrices which is commutative. In order to do this we 
	will be motivated to add a symmetry conditions, called bi-symmetry, that characterizes 
	the sub-ring of bi-symmetric X-matrices. 
	
	Because of the determinant structure, the characteristic polynomial of an X-matrix is 
	a formed by products of quadratic polynomials. This structure leads to an eigenvalue 
	inclusion result for this type of matrices that shows that the eigenvalues of X-matrices 
	lie exactly on the boundary of the Cassini ovals.
	
	The problem of locating the spectrum of a square matrix $A$ 
	without actually solving $\det(sI-A)=0$ has attracted much attention in numerical 
	linear algebra. The study of diagonally dominant matrices leads 
    to Gershgorin's disk theorem, refining the argument gives 
	algebraic curves encircling the matrix spectrum like Cassini's ovals of 
    Brauer or the lemniscates of Brualdi. From another viewpoint, 
	the convex hull of the Gershgorin disks coincides with
    the numerical range (or the field of values) 
	of a matrix with respect to the $\infty$-norm, 
	the set of all Rayleigh quotients obtained from dual pairs of the $\infty$-norm.
	The largest real part of the numerical range is an example of a matrix norm that is not an operator norm.
	
	Finally, we show that, for any monic polynomial on $ \mathbb{R} $ there exists a 
	corresponding X-shaped companion matrix. Thus, an X-matrix can be associated with any 
	polynomial in $ \R $, a feature which does not hold for diagonal matrices.

We note that a member of this family appears naturally in the work of \cite[Equation (3), p. 4010)]{OwenVari20}  in association with the fitting of a two-factor regression model with interactions.

	\section{Notation and preliminary remarks}
	We denote by $\K$ a field, e.g.~$\K= \Q$, $\R$, $\C$. $\Mat(n,\K)$ is the ring of 
	square matrices 
	$n\times n$ with elements in $\K$. 
	We recall that the operation of transposition turns a matrix $A=(a_{i,j})$ into 
	$A^T=(a_{j,i})$, 
	and satisfies the relation $(AB)^T=B^T A^T$. Here we define the anti-transpose of 
	a matrix and 
	show a basic property related to this operation.
	\begin{definition}\label{def:antitrans}
		Let $A=(a_{i,j}) \in \Mat(n, \K)$. We define the \emph{anti-transpose} of $A$, 
		denoted by $A^{\AT}$, the matrix whose general element is given by
		\[
		(A^{\AT})_{i,j}=a_{n-j+1, n-i+1}.
		\]
	\end{definition}
	\begin{example}
		Given 
		$A=
		\begin{bmatrix}
		1 & 2 & 3 \\
		4 & 5 & 6 \\
		7 & 8 & 9 \\
		\end{bmatrix}$, its anti-transpose is
		$A^{\AT}=
		\begin{bmatrix}
		9 & 6 & 3 \\
		8 & 5 & 2 \\
		7 & 4 & 1 \\
		\end{bmatrix}.$		
	\end{example}
	Let $J$ denote the \emph{anti-identity} in 
	$\Mat(n,\K)$, i.e.
	\[
	J=(j_{i,j})_{i,j=1,\dots,n}=
	\begin{cases}
	1 &\text{if} \ j=n-i+1 \\
	0& \text{otherwise}
	\end{cases}.
	\]
	\begin{lemma}
		The antitranspose and transpose of a matrix are related by $A^{\AT}=(JAJ)^T$.
	\end{lemma}
	\begin{proof}
		Left-multiplication with $J$ reorders the rows from last to first and 
		right-multiplication with $J$
		reorders the columns from last to first, transposition exchanges the row and column 
		indices.
	\end{proof}
	These two immediate consequences follow.
	\begin{corollary}\label{cor:sym}
		If $A$ is symmetric then its anti-transpose is symmetric as well. 
		Moreover the equality $AJ=JA^{\AT}$ holds.
	\end{corollary}
	\begin{corollary}\label{prop:antitransposeprod}
		For $A,B\in \Mat(n,\K)$ it holds that $(AB)^{\AT}=B^{\AT} A^{\AT}$.
	\end{corollary}
	\begin{proof}
		The anti-identity $J$ is involutory, namely $J^2=I$, so that
		\[
		(AB)^{\AT}=(JABJ)^T=(JAJJBJ)^T=(JBJ)^T(JAJ)^T=B^{\AT} A^{\AT}.\qedhere
		\]
	\end{proof}
	
	
	\section{X-matrices}
	\begin{definition}
		We say that a matrix $A=(a_{i,j}) \in \Mat(n, \K)$ is an \emph{X-matrix} if
		\begin{equation}\label{X}
		a_{i,j}=0 \quad \text{whenever} \ i \neq j \ \text{and} \ i+j \neq n+1.
		\end{equation}
		We denote by $\X$ the subset of $\Mat(n,\K)$ whose elements are the X-matrices.
	\end{definition}
	\noindent In other words, a square matrix is an X-matrix if each element which 
	neither located on the diagonal nor on the anti-diagonal is 0.
	\begin{example}
		The matrices
		\[
		\begin{bmatrix}
		1 & 0 & 0 & 2 \\
		0 & 0 & 1 & 0 \\
		0 & 4 & 9 & 0 \\
		3 & 0 & 0 & 0 \\
		\end{bmatrix},
		\quad
		\begin{bmatrix}
		1 & 0 & 0 & 0 & 5 \\
		0 & 4 & 0 & 3 & 0 \\
		0 & 0 & 6 & 0 & 0 \\
		0 & 9 & 0 & 1 & 0 \\
		7 & 0 & 0 & 0 & 8 \\
		\end{bmatrix}
		\]
		are $4\times 4$ and $5\times 5$ X-matrices, respectively.
	\end{example}
	\noindent We are now going to evidence some first basic properties of X-matrices, 
	which motivate us in studying these nice-shaped objects.
	\begin{remark}\label{doublediag}
		Any X-matrix $A$ can be written as $A_1 I + A_2 J$ where $A_1, A_2$ are diagonal matrices, 
        $I$ is the identity matrix and $J$ is the anti-identity defined above. 
        Moreover, if $n$ is even (i.e.~the matrix has no central element), $A_1$ and $A_2$ are uniquely determined. 
        Conversely, if $A_1$ and $A_2$ are diagonal matrices, then $A=A_1 I + A_2 J$ is an X-matrix.
	\end{remark}
	\begin{lemma}
		If $A \in \X$, then $A^T \in \X$, as well as $A^{\AT} \in \X$.
	\end{lemma}
	\begin{proof}
		Geometrically speaking, it should be clear that the statement is true looking at the examples above. Anyway, assume that $A \in \X$. Then
		\[
		(A^{T})_{i,j}=a_{j,i}
		\]
		is non-zero whenever $j \neq i$ and $j+i\neq n+1$. Therefore, $A^T \in \X$. Similarly,
		\[
		(A^{\AT})_{i,j}=a_{n-j+1,n-i+1}
		\]
		is non-zero if $n-j+1 \neq n-i+1$ and $n-j+1 + n-i+1 \neq n+1$. These conditions are equivalent to $i\neq j$ and $i+j \neq n+1$; thus, $A^{\AT} \in \X$.
	\end{proof}
	\begin{lemma}\label{Le:Prod}
		Assume that $A, B \in \X$. Then $A+B$, $-A$, $AB$ are elements of $\X$.
	\end{lemma}
	\begin{proof}
		The fact that $A+B$ and $-A$ belong to $\X$ follows immediately by the definition of matrix sum. In order to prove that $AB$ belongs to $\X$, use the Remark \ref{doublediag} to write 
		\[
		A=A_1 I + A_2 J, \quad B=B_1 I + B_2 J,
		\]
		where $A_1, A_2, B_1, B_2$ are diagonal matrices. Then, by Corollary \ref{cor:sym},
		\[
		AB=(A_1 I + A_2 J)(B_1 I + B_2 J)=(A_1 B_1 + A_2 B_2^{\AT})I + (A_2 B_1^{\AT} + A_1 B_2) J
		\]
		is still an X-matrix.
	\end{proof}
	
	\begin{proposition}
		$\X$ is a subring of $\Mat(n,\K)$ which is closed under transposition and anti-transposition.
	\end{proposition}
	\begin{proof}
		The result comes from the previous lemmas, in addition to the fact that the zero matrix and the identity matrix clearly belong to $\X$.
	\end{proof}
	The above analysis shows that X-matrices are invariant under a number of matrix operations, including sum and product.
	\begin{corollary}
		Let $ m\in \mathbb{N} $. If $A\in \X$, then $ A^m\in \X $.
	\end{corollary}
	\begin{proof}
		Follows by Lemma \ref{Le:Prod} and by the fact that $ A^m$ is the product of $ A $ with itself $ m $ times. 
	\end{proof}
	We then have the following.
	\begin{proposition}
		For $A \in \mathcal{X}$, any matrix function
		\[
		f(A)=\sum_{i=0}^{\infty} a_i A^i
		\]
		defined via a matrix Taylor series which converges in $A$ belongs to $\mathcal{X}$.
	\end{proposition}
	\begin{proof}
		$f(A)$ is the limit of the evaluations at $A$ of polynomials with coefficients in $\mathbb{K}$; each of these evaluations is an element of $\mathcal{X}$ by the previous results. Therefore $f(A)$ belongs to $\mathcal{X}$, as any vector subspace of a finite dimensional vector space is closed.
	\end{proof}
	
	These functions include function such as matrix exponentials, logarithms, harmonic functions and the Cayley transformation $\psi(A)=(I-A)(I+A)^{-1}$.
	
	
	\section{The inverse of an X-matrix}
	This section is split into two parts. In the first, we characterize the invertible elements of $\X$, providing an explicit formula for the discriminant of an X-matrix. In the second, we determine a way to compute the inverse of an X-matrix conveniently.
	
	\subsection{Characterizing the Elements of an Inverse X-Matrix}
	Having proved that $\X$ is a subring of $\Mat(n,\K)$ one could wonder whether for any non-zero element of $\X$ it is possible to find a multiplicative inverse inside $\X$ itself. It is easy to find a non-zero X-matrix which is not invertible. For instance, one can consider the matrix
	\[
	\begin{bmatrix}
	1 & 0 & \dots & 0\\
	0 & 0 & & \\
	\vdots & & \ddots & \vdots \\
	0 & & \dots & 0
	\end{bmatrix}.
	\]
	Then, we want to understand what being invertible in $\X$ means, i.e. to characterize the set
	\[
	\X^\times = \set{ A \in \X | \ \text{there exists} \ B\in \X \ \text{ such that } \ AB=BA=I }.
	\]
	In this respect, observe that, since $\X$ is a subring of $\Mat(n,\K)$,
	if a matrix $A\in \X$ is invertible in $\X$, then it is also invertible in $\Mat(n,\K)$, or equivalently $\det A \neq 0$. This is equivalent to say that $\X^\times \subseteq \GL(n,\K)^\times \cap \X$.
	\medskip
	We are going to prove that the inverse of an X-matrix is an X-matrix. Because a square matrix is invertible if and only if $ \det A \neq 0 $, we start characterizing the determinant of an X-matrix. 
	\begin{lemma} \label{Le:Det}
		Given $ A\in \X $,
		\begin{equation}\label{eqdet}
		\det A=
		\begin{cases}
		\prod_{i=1}^{\frac{n}{2}} (a_{i,i}a_{n-i+1,n-i+1}-a_{i,n-i+1}a_{n-i+1,i}) & \text{if} \ $n$ \ \text{is even,} \\
		\prod_{i=1}^{\frac{n-1}{2}} (a_{i,i}a_{n-i+1,n-i+1}-a_{i,n-i+1}a_{n-i+1,i})\cdot a_{(n+1)/2,(n+1)/2} & \text{if} \ $n$ \ \text{is odd.}
		\end{cases}
		\end{equation}
	\end{lemma}
	
	\begin{proof}
		The proof is by induction. The cases $n=1,2$ are evident. Therefore, assume that $n \ge 3$. In order to prove the claim we observe that each X-matrix of $\Mat(n,\K)$ is of the type 
		\[
		A=
		\begin{bmatrix}
		a_{1,1} & 0 & \dots & 0 & a_{1,n} \\
		0 & \quad & \quad & \quad & 0 \\
		\vdots & \quad & A_{n-2} & \quad & \vdots \\
		0 & \quad & \quad & \quad & 0 \\
		a_{n,1} & 0 & \dots & 0 & a_{n,n} \\
		\end{bmatrix}
		\]
		where $A_{n-2}$ is another X-matrix of $\Mat(n-2, \K)$. Expanding the determinant with the Laplace rule on the first column we obtain
		\[
		\det A = a_{1,1} \cdot
		\begin{vmatrix}
		& & & 0 \\
		& A_{n-2} & & \vdots \\
		& & & 0 \\
		0 & \dots & 0 & a_{n,n} \\
		\end{vmatrix}
		+ (-)^{n+1} a_{n,1} \cdot
		\begin{vmatrix}
		0 & \dots & 0 & a_{1,n} \\
		& & & 0 \\
		& A_{n-2} & & \vdots \\
		& & & 0 \\
		\end{vmatrix}.
		\]
		Using another time the Laplace rule we have
		\begin{align*}
		\det A = a_{1,1} \cdot a_{n,n} \cdot \det A_{n-2}+ (-)^{n+1}(-)^{n} a_{n,1} &\cdot a_{1,n} \cdot \det A_{n-2} = \\ &=(a_{1,1} a_{n,n} - a_{n,1} a_{1,n} )\cdot \det A_{n-2},
		\end{align*}
		and by induction we obtain the equality \eqref{eqdet}.
	\end{proof}
	\begin{remark}
	One can note that the determinant of a $3\times 3$ X-matrix is the difference between the product of the elements on the diagonal and the one of the elements on the anti-diagonal:
	\[
	\det
	\begin{bmatrix}
		a & 0 & d \\
		0 & b & 0 \\
		e & 0 & c \\
	\end{bmatrix}
	=
	abc-dbe.
	\] 
\end{remark}
	\begin{remark}[The characteristic polynomial of an X-matrix] \label{Char:X}
		So far we did not use the fact the ring from which X-matrices take elements is a field. Hence, the formula for the determinant of an X-matrix holds for X-matrices of $\Mat(n,R)$, where $R$ is a commutative ring. 
		Using this fact applied to the ring $R=\K\lbrack{}t\rbrack{}$, we can easily find the formula for the characteristic polynomial of an X-matrix. We recall that, given $A\in \Mat(n,\K)$, $\chi_A(t)=\det(tI-A)\in \K\lbrack{}t\rbrack{}$ is the characteristic polynomial. If $A\in\X$, then $tI-A$ is an X-matrix of $\Mat(n,\K\lbrack{}t\rbrack{})$. By Lemma \ref{Le:Det}, the characteristic polynomial of an X-matrix can be written as:
		\begin{equation*}
		\chi_A(t)=(t-a_{\frac{n+1}{2},\frac{n+1}{2}})\prod_{i=1}^{\frac{n-1}{2}} [(t-a_{i,i})(t-a_{n-i+1,n-i+1})-a_{i,n-i+1}a_{n-i+1,i}]
		\end{equation*}
		for $n$ odd, or
		\begin{equation*}
		\chi_A(t)=\prod_{i=1}^{\frac{n}{2}} [(t-a_{i,i})(t-a_{n-i+1,n-i+1})-a_{i,n-i+1}a_{n-i+1,i}]
		\end{equation*}
		for $n$ even.
		
	\end{remark}
	
	\begin{proposition}\label{prop:inverse}
		$\X^\times = \GL(n,\K) \cap \X$.
	\end{proposition}
	
	\begin{remark}
		While $R^\times\subseteq \GL(n,\K) \cap R$ holds for any subring $R\subseteq \Mat(n,\K)$, we cannot say the same about the other inclusion. For example, if $\K=\Q$, let $R$ be $\Mat(n,\Z)$, which is a subring of $\Mat(n,\Q)$. 
		In this case the matrix $A=2I \in \GL(n,\Q)$ (since $\det A = 2^n \neq 0$) and $A \in R$ (because its elements lie in $\Z$). However, it is not true that $A \in R^\times$, since $A^{-1}$, the only matrix in $\Mat(n,\Q)$ such that $AA^{-1}=A^{-1}A=I$, is $A^{-1}=\frac{1}{2} I\notin R$, since not all its elements are integers.
		Thus, the equality between sets that we are going to prove is a property of the particular subring $\X$, that holds thanks to how it is constructed.
	\end{remark}
	\begin{proof}[Proof of Proposition \ref{prop:inverse}.]
		We have to prove that $\GL(n,\K{}) \cap \X \subseteq \X^\times$, or equivalently that, if $A\in \X$ and $\det A \neq 0$, then $A^{-1}\in \X$. We proceed as follows. Firstly, we recall that the inverse of a matrix $A$ is unique and that $A^{-1}A=AA^{-1}=I$. Thus, if we were able to find an element $B \in \mathcal{X}$ such that $AB=BA=I$, that element would be the unique inverse of $A$. Then, let us consider a generic $B \in \mathcal{X}$ and set $C=AB$. Note that the elements of a product matrix $C=AB$ with $A\in\X  $ and $ B\in \X $ are of the type:
		\begin{itemize}
			\item If $n$ is even,
			\begin{equation}\label{coeffprodeven}
			c_{i,j}=
			\begin{cases}
			a_{i,i}b_{i,i}+a_{i,n-i+1}b_{n-i+1,i} & \text{if} \ i=j \\
			a_{i,i}b_{i,n-i+1}+a_{i,n-i+1}b_{n-i+1,n-i+1} & \text{if} \ i+j=n+1 \\
			0 & \text{otherwise}
			\end{cases}\qquad.
			\end{equation}
			\item If $n$ is odd,
			\begin{equation}\label{coeffprododd}
			c_{i,j}=
			\begin{cases}
			a_{i,i}b_{i,i}+a_{i,n-i+1}b_{n-i+1,i} & \text{if} \ i=j\neq \frac{n+1}{2} \\
			a_{i,i}b_{i,n-i+1}+a_{i,n-i+1}b_{n-i+1,n-i+1} & \text{if} \ i+j=n+1, \ i, j \neq \frac{n+1}{2} \\
			a_{(n+1)/2,(n+1)/2} b_{(n+1)/2,(n+1)/2} & \text{if} \ i=j=\frac{n+1}{2} \\
			0 & \text{otherwise}
			\end{cases}\qquad.
			\end{equation}
		\end{itemize}
		Then, we can impose that all the elements on the diagonal are equal to $1$ and all the 
		elements on the anti-diagonal are equal to $0$:
		\begin{itemize}
			\item If $n$ is even,
			\begin{equation*}
			\begin{cases}
			a_{i,i}b_{i,i}+a_{i,n-i+1}b_{n-i+1,i}=1 & \text{for} \ i=1,2,\dots,n\\ 
			a_{i,i}b_{i,n-i+1}+a_{i,n-i+1}b_{n-i+1,n-i+1}=0 & \text{for} \ i=1,2,\dots,n\\
			\end{cases}\quad,
			\end{equation*}
			which is equivalent to
			\begin{equation*}
			\begin{bmatrix}
			a_{i,i} & a_{i,n-i+1} \\
			a_{n-i+1,i} & a_{n-i+1,n-i+1} \\
			\end{bmatrix}
			\begin{bmatrix}
			b_{i,i} & b_{i,n-i+1} \\
			b_{n-i+1,i} & b_{n-i+1,n-i+1} \\
			\end{bmatrix}
			=
			\begin{bmatrix}
			1 & 0 \\
			0 & 1 \\
			\end{bmatrix}
			\quad \text{for} \ i=1, \dots, \frac{n}{2}.
			\end{equation*}
			The idea of what we are doing is to focus on the $i$-th and $n-i+1$-th rows and columns 
			at the same time. We notice that the condition $a_{i,i}a_{n-i+1,n-i+1}-a_{i,n-i+1}a_{n-i+1,i}\neq 0$ 
			easily follows by the hypothesis $A \in \GL(n,\K)$, i.e.~$\det A \neq 0$, and by the equation
			\eqref{eqdet} for the determinant of an X-matrix. Thus, inverting the $2\times 2$ submatrices of $A$, we can find a suitable matrix $B \in \X$ such that $AB=I$.
			\item If $n$ is odd,
			\begin{equation*}
			\begin{cases}
			a_{i,i}b_{i,i}+a_{i,n-i+1}b_{n-i+1,i}=1 & \text{for} \ i=1,2,\dots,n, \ i \neq \frac{n+1}{2}\\ 
			a_{i,i}b_{i,n-i+1}+a_{i,n-i+1}b_{n-i+1,n-i+1}=0 & \text{for} \ i=1,2,\dots,n, \ i \neq \frac{n+1}{2}\\
			a_{(n+1)/2,(n+1)/2}b_{(n+1)/2,(n+1)/2}=1
			\end{cases}\quad,
			\end{equation*}
			which is equivalent to
			\begin{equation*}
			\begin{bmatrix}
			a_{i,i} & a_{i,n-i+1} \\
			a_{n-i+1,i} & a_{n-i+1,n-i+1} \\
			\end{bmatrix}
			\begin{bmatrix}
			b_{i,i} & b_{i,n-i+1} \\
			b_{n-i+1,i} & b_{n-i+1,n-i+1} \\
			\end{bmatrix}
			=
			\begin{bmatrix}
			1 & 0 \\
			0 & 1 \\
			\end{bmatrix}
			\quad \text{for} \ i=1, \dots, \frac{n-1}{2},
			\end{equation*}
			together with the equation
			\[
			a_{(n+1)/2,(n+1)/2}b_{(n+1)/2,(n+1)/2}=1.
			\]
			Now the argument is exactly the same as above, once noticed that also the condition $a_{(n+1)/2,(n+1)/2} \neq 0$ holds.
		\end{itemize}
		In both cases, we found $B \in \X$ such that $AB=I$. A well known argument shows that $BA=I$, 
		too: since $AB=I$, we have that $\det A \det B = 1$, so that $\det B \neq 0$ and $B^{-1}$ exists.
		Thus,
		\[
		BA=BABB^{-1}=B (AB) B^{-1} = BIB^{-1}=BB^{-1}=I.
		\]
		In particular, $B \in \X$ is the inverse of $A$ and the claim is proven.
	\end{proof}
	
	\subsection{Finding the Inverse of an X-matrix}
	We propose another approach to find the inverse of an X-matrix, taking an empirical approach. As we observed in Remark \ref{doublediag}, an X-matrix is the the sum of a diagonal and an anti-diagonal term; notationwise, in this section, we write $A=DI+EJ$ where $D$ and $E$ are diagonal matrices, $I$ is the identity and $J$ is the anti-identity, i.e.~the matrix made by ones on the anti-diagonal, and zeros everywhere else. Observe that the choice of $D, E$ is unique whenever $n$ is even, while it is not if $n$ is odd (the diagonal and the anti-diagonal intersect in the \emph{center} of the matrix). In the second case, we impose $e_{(n+1)/2,(n+1)/2}=0$, so that $d_{(n+1)/2,(n+1)/2}=a_{(n+1)/2,(n+1)/2}$. Assume that the matrix $D$ is invertible, i.e.~$a_{i,i}=d_{i,i}\neq 0$ for all $i=1, \dots, n$. Then $D^{-1}A=I+FJ$, where $F=\diag(\frac{e_{i,i}}{d_{i,i}})$ is the diagonal of quotients.
	\medskip
	
	\noindent We are about to use the expression of the geometric series in an \emph{heuristic} way, without caring (at least for now) about the convergence of the series that follows. Recall that, by Corollary \ref{cor:sym}, for $F$ diagonal, we have that $JF=F^{\AT}J$. Therefore
	\begin{multline*}
	A^{-1}D=I-FJ+(FJ)^2-(FJ)^3+(FJ)^4-(FJ)^5\pm\dots =\\
	=I-FJ+FF^{\AT}J^2-FF^{\AT}FJ^3 + FF^{\AT}F F^{\AT} J^4 - FF^{\AT}F F^{\AT} F J^5\pm \dots =\\
	=I-FJ+FF^{\AT} J^2-(F F^{\AT}) F J^3 + (FF^{\AT})^2 J^4 - (FF^{\AT})^2 F J^5\pm \dots =\\
	=(I+FF^{\AT}+ (FF^{\AT})^2+ \dots)I -(I+ FF^{\AT} +(FF^{\AT})^2+\dots)FJ  =\\
	= (I-FF^{\AT})^{-1} I - ( I-FF^{\AT})^{-1}FJ =  (I-FF^{\AT})^{-1} (I- FJ).
	\end{multline*}
	Hence
	\[
	A^{-1}=(I-FF^{\AT})^{-1} (I- FJ) D^{-1}.
	\]
	For this formula to hold, we need $a_{i,i}=d_{i,i}\neq 0$ for all $i=1,\dots,n$ and $e_{i,i} e_{n-i,n-i}\neq d_{i,i} d_{n-i,n-i}$. The obtained formula seems to hold only if the geometric series
	\[
	(I+FJ)^{-1}=I-FJ+(FJ)^2-(FJ)^3+(FJ)^4-(FJ)^5\pm\dots
	\]
	converges. Anyway, we can verify \emph{by hands} that the formula holds even if the matrix $FJ$ has elements big enough not to let the series converge:
	\begin{multline*}
	(I-FF^{\AT})^{-1} (I- FJ) D^{-1} A = (I-FF^{\AT})^{-1} (I- FJ) D^{-1} (DI+EJ)= \\
	= (I-FF^{\AT})^{-1} (I- FJ) (I+FJ)=(I-FF^{\AT})^{-1} (I- (FJ)^2) =\\
	=(I-FF^{\AT})^{-1} (I-FF^{\AT})=I,
	\end{multline*}
	and necessarily $A (I-FF^{\AT})^{-1} (I- FJ) D^{-1}=I$, too.
	\medskip
	
	\noindent Alternatively, assume $E$ invertible (for, is $n$ is odd impose $d_{(n+1)/2,(n+1)/2}=0$, so that $e_{(n+1)/2,(n+1)/2}=a_{(n+1)/2,(n+1)/2}$). Using $G=\diag(\frac{d_{i,i}}{e_{i,i}})$, we find $E^{-1}A=J+GI=(I+GJ)J$ and
	\begin{multline*}
	J A^{-1}E=I-GJ+GG^{\AT}J^2-GG^{\AT}GJ^3 + GG^{\AT}G G^{\AT} J^4 - GG^{\AT}G G^{\AT} G J^5\pm \dots \\
	= (I-GG^{\AT})^{-1} (I- GJ)
	\end{multline*}
	so that
	\[
	A^{-1}= J (I-GG^{\AT})^{-1} (I- GJ) E.
	\]
	For this formula to hold, we need $a_{i,i}=e_{i,i}\neq 0$ and $e_{i,i} e_{n-i,n-i}\neq d_{i,i} d_{n-i,n-i}$.
	\medskip
	
	We remark that these formulas are useful since, in both cases, the only matrices that we are inverting are diagonal matrices, which is a very easy task to accomplish. As already noticed, if $n$ is odd then the central entry is not uniquely decomposable. One can avoid to impose one of $e_{(n+1)/2,(n+1)/2}$ and $d_{(n+1)/2,(n+1)/2}$ equal to 0, as long as the required conditions are satisfied.
	
	\section{The bi-symmetry condition}
	We have already observed that $\X$ is not a commutative ring whenever $n\ge2$. However, it would be interesting to determine at least a subring of $\X$ which is commutative. There is a canonical way to do so: given a (possibly non-commutative) ring $\mathcal{A}$, one can consider the \emph{center} of $\mathcal{A}$, defined by
	\[
	\mathcal{Z}(\mathcal{A})=\set{a \in \mathcal{A} | ab=ba, \ \text{for all} \ b \in \mathcal{A}}.
	\]
	One can easily prove that $\mathcal{Z}(\mathcal{A})$ is a subring of $\mathcal{A}$ which, by definition, is commutative. However, in our particular case, this approach produces something already known: Indeed, it can be shown that the center of $\X$ is given by the diagonal matrices which are invariant under anti-transposition. For example, in the cases $n=3$ and $n=4$, $\mathcal{Z}(\X)$ consists of the matrices of the shape
	\[
	\begin{bmatrix}
	a & 0 & 0 \\
	0 & b & 0 \\
	0 & 0 & a \\
	\end{bmatrix},
	\quad
	\begin{bmatrix}
	a &  & \dots & 0 \\
	& b &  & \vdots \\
	\vdots &  & b &  \\
	0 & \cdots &  & a \\
	\end{bmatrix}.
	\]
	\noindent As we have already said, we are looking for something wider than a subset of diagonal matrices. Then the question is whether it is possible to impose a particular symmetry condition on the elements of $\X$ to obtain the result we look for.
	\begin{definition}
		We denote by $\X^b$ the subset of $\X$ whose elements are the X-matrices invariant under both transposition and anti-transposition, i.e.~
		\[
		\X^b=\set{ A \in \X | A=A^T=A^{\AT} }.
		\]
		We call the elements of $ \X^b $ \emph{bi-symmetric X-matrices}.
	\end{definition}
	\begin{example}
		For $n=3$ and $n=4$ we have that bi-symmetric X-matrices are have shapes
		\[
		\begin{bmatrix}
		a & 0 & c \\
		0 & b & 0 \\
		c & 0 & a \\
		\end{bmatrix},
		\quad
		\begin{bmatrix}
		a & 0 & 0 & c \\
		0 & b & d & 0 \\
		0 & d & b & 0 \\
		c & 0 & 0 & a \\
		\end{bmatrix}\text{, respectively.}
		\]
	\end{example}
	\noindent In particular, given $A \in \X^b$ invariance under anti-transposition imposes
	\begin{equation}\label{antitran}
	a_{i,i}=a_{n-i+1,n-i+1} \quad \text{for all} \ i=1,\dots,n,
	\end{equation}
	while invariance under transposition imposes
	\begin{equation}\label{tran}
	a_{i,n-i+1}=a_{n-i+1,i} \quad \text{for all} \ i=1,\dots,n.
	\end{equation}
	\begin{proposition}\label{Xbsr}
		$\X^b$ is a commutative subring of $\X$.
	\end{proposition}
	\begin{proof}
		It is clear that the zero matrix belongs to $\X^b$, and that $\X^b$ is closed with respect to the sum and the additive inverse operations. Thus, $\X^b$ is an additive subgroup of $\X$. Moreover, note that the identity matrix lies inside $\X^b$. In order to prove that $\X^b$ is a subring of $\X$, what is left to show is that $\X^b$ is closed under the product operation. Assume that we are given $A, B \in \X^b$. Set $C=AB$; we need to prove that also for $C$ conditions (\ref{antitran}) and (\ref{tran}) hold. Let us start from (\ref{antitran}). From equations (\ref{coeffprodeven}) and (\ref{coeffprododd}), for all off-center elements we have
		\[
		c_{i,i}=a_{i,i}b_{i,i}+a_{i,n-i+1}b_{n-i+1,i}=a_{n-i+1,n-i+1}b_{n-i+1,n-i+1}+a_{n-i+1,i}b_{i,n-i+1}.
		\]
		Note that $a_{n-i+1,n-(n-i+1)+1}=a_{n-i+1,i}$, as well as $b_{n-(n-i+1)+1,n-i+1}=b_{i,n-i+1}$. Hence,
		\[
		c_{i,i}=a_{n-i+1,n-i+1}b_{n-i+1,n-i+1}+a_{n-i+1,n-(n-i+1)+1}b_{n-(n-i+1)+1,n-i+1}=c_{n-i+1,n-i+1}.
		\]
		If $n$ is odd and $i=\frac{n+1}{2}$,
		\[
		c_{(n+1)/2,(n+1)/2}=c_{n-((n+1)/2)-1,n-((n+1)/2)-1},
		\]
		so that this peculiar case is verified, too. This proves that $C=C^{\AT}$. Now we pause this argument for a moment, and we use the results obtained so far to show that, if $A, B \in \X^b$, then $AB=BA$. Recalling the remarks made at the very beginning of this paper about the properties of anti-transposition, we deudce that
		\[
		AB=(AB)^{\AT}=B^{\AT} A^{\AT}=BA.
		\]
		Lastly, we can go back and use the commutative property to show that $C=C^T$:
		\[
		C^T= (AB)^T=B^T A^T=BA=AB=C,
		\]
		and this proves the claim.
	\end{proof}
	We can investigate the commutativity of the ring $\X^b$ further. Indeed, one can observe that bi-symmetric matrices admit a common basis of eigenvectors (in particular, $e_i \pm e_{n-i+1}$ for $i=1,\dots, n/2$ if $n$ is even, and $e_i \pm e_{n-i+1}$ for $i=1, \dots, (n+1)/2$, together with $e_{(n+1)/2}$, if $n$ is odd). Let $C$ be the $n\times n$ invertible matrix obtained by collecting these vectors. Then the ring homomorphism given by the conjugation
	\[
	\X^b \to \Mat(n,\K), \quad A \mapsto C^{-1}AC
	\]
	is injective and has image contained inside the diagonal matrices. Thus, $\X^b$ is commutative being isomorphic to a subring of the diagonal matrices, which constitute a commutative ring. This is a specific situation of a general result:
	\begin{proposition}
		Assume that $\mathcal{Y}$ is a subring of $\Mat(n,\K)$ such that
		\begin{itemize}
			\item For any $A \in \mathcal{Y}$, $A$ is nondefective;
			\item For any $A \in \mathcal{Y}$ and $\lambda \in \K$, $\lambda A \in \mathcal{Y}$ (so that $\mathcal{Y}$ is a $\K$-vector subspace of $\Mat(n,\K)$).
		\end{itemize}
		Then $\mathcal{Y}$ is commutative if and only if there is a common basis of eigenvectors for all the elements of $\mathcal{Y}$.
	\end{proposition}
	\begin{proof}
		Indeed, if there is a common basis of eigenvectors, the argument is the same as above: pick $C$ the invertible matrix obtained by collecting together these eigenvectors. The injective morphism
		\[
		\mathcal{Y} \to \Mat(n,\K), \quad A \mapsto C^{-1}AC
		\]
		maps $\mathcal{Y}$ into a subring of the diagonal matrices, which is commutative and isomorphic to $\mathcal{Y}$. On the other hand, assume $\mathcal{Y}$ is commutative. Choose a basis $\Set{y_1, \dots, y_r}$ of $\mathcal{Y}$. Since these elements are diagonalizable and commute with each other, we can find a common basis of eigenvectors $\Set{x_1, \dots, x_n}$. This is a general fact of linear algebra, and follows from the fact that if $AB=BA$ then $B$ sends an eigenvector of $A$ to an eigenvector of $A$ relative to the same eigenvalue. Hence, we have that
		\[
		y_i x_j = \lambda_{ij} x_j
		\]
		for some $\lambda_{ij} \in \K$. Now pick a generic element $\alpha_1 y_1 + \dots + \alpha_r y_r \in \mathcal{Y}$; then
		\[
		(\alpha_1 y_1 + \dots + \alpha_r y_r)x_j= (\alpha_1 \lambda_{1j} + \dots + \alpha_r \lambda_{rj})x_j,
		\]
		so that $\Set{x_1,\dots,x_n}$ is a common basis of eigenvectors for the whole $\mathcal{Y}$.
	\end{proof}
	We now show that if an X-matrix is bi-symmetric and invertible, then also its inverse is bi-symmetric and invertible.
	\begin{proposition}
		Assume that $A$ is a bi-symmetric invertible matrix. Then $A^{-1}$ is a bi-symmetric matrix.
	\end{proposition}
	\begin{proof}
		We already know by Proposition \ref{prop:inverse} that $A^{-1} \in \X$. Transposing and anti-transposing the equation $A A^{-1}=I$ we obtain
		\[
		(A^{-1})^T A^T=I^T=I, \quad (A^{-1})^{\AT} A^{\AT}=I^{\AT}=I.
		\]
		By uniqueness of the inverse of a matrix, we deduce that
		\[
		(A^{-1})^T=(A^T)^{-1}=A^{-1}, \quad (A^{-1})^{\AT}=(A^{\AT})^{-1}=A^{-1}.
		\]
		Thus, $A^{-1} \in \X^b$.
	\end{proof}
	\noindent We can summarize this results in the following statement:
	\[
	(\X^b)^\times = \GL(n,\K) \cap \X^b.
	\]
	
	
	\section{X-Companion Matrices and Eigenvalue Inclusion}
	In linear algebra, companion matrices and eigenvalue inclusion play a central role and are subject to intensive investigation (see \cite{Key2004103} and more recently \cite{Deaett2019223,Chan2017335} on companion matrices, and works such as \cite{Brual82,Hadji14,Bu2015} on eigenvalue inclusion).
	In this section,  we discuss whether X-matrices can be companion on any monic polynomial on $ \mathbb{K} $ in the first part. We present results for eigenvalue inclusion in the second part.
	
	\subsection{Companion Matrices in X-form}
	We look for a non-trivial answer to  the question of whether it might be true that, given a monic polynomial, there always exists an X-matrix whose characteristic polynomial is $f$. 
	We recall that a monic polynomial $f(t) \in \K[t]$ is an expression of the type $f(t)=t^n+a_{n-1}t^{n-1}+\dots+a_0$. We call the \emph{companion matrix} of $f$ the $n\times n$ matrix $ C $ given by
	\[
	C=\begin{bmatrix}
	0 & 0 & \dots & 0 & -a_0 \\
	1 & 0 & \dots & 0 & -a_1 \\
	0 & 1 & \dots & 0 & -a_2 \\
	\vdots & \vdots & \ddots & \vdots & \vdots \\
	0 & 0 & \dots & 1 & -a_{n-1} \\
	\end{bmatrix}.
	\]
	The characteristic polynomial of $ C $ coincides with $f$. A well known result is that whenever $\mathbb{K}$ is algebraically closed (e.g.~$\mathbb{K}=\C$), each polynomial factors into linear terms as
	\[
	f(t)=\prod_{i=1}^n (t-\alpha_i),
	\]
	so that there exists a diagonal matrix that satisfies the property. Because diagonal matrices are a subset of $X $-matrices, we have immediately a trivial answer to this section starting question. On the other hand, this is not true for other fields.
	
	For instance, $f(t)=t^3-2$ is irreducible over $\Q$ (since it is of degree 3 and without roots in $\Q$). We have already remarked that, whenever $n$ is odd, any $n\times n$ matrix has at least an eigenvalue, i.e.~the characteristic polynomial has at least one root. Therefore, there cannot exist a $3\times 3$ X-matrix for such an $f$ on $ \Q$.
	
	In the remainder, we deal with the very familiar case $\K=\R$. In this case, in spite of what happens in $\C$, the field is not algebraically closed, hence diagonal matrices are not enough to cover all the monic polynomials of $\R[t]$ by taking their characteristic. However, all the monic polynomials of $\R[t]$ can be covered by X-matrices.
	\begin{theorem}
		Given any $f(t)\in \R[t]$ a monic polynomial of degree $n$ with coefficients in $\R$, there exists a $n\times n$ X-matrix $A\in\Mat(n,\R)$ such that $\chi_A(t)=f(t)$.
	\end{theorem}
	\begin{proof}
		Let $f(t)$ be a monic polynomial of degree $n$, such that $f\in \R [t]$. It is known that every polynomial with coefficients in $\R$ can be factorized into polynomials whose degree is at most $2$. Hence we can write
		\[
		f(t)=\prod_{i=1}^{k}(t-\alpha_i) \cdot \prod_{i=k+1}^l (t^2+\beta_i t +\gamma_i) \qquad \alpha_i,\beta_i,\gamma_i\in\R.
		\]
		That is, $f$ is the product of $l$ polynomials, of which $k$ have degree $1$ and the remaining $l-k$ have degree $2$, with $k+2(l-k)=n$ (if $n$ is odd, $k>0$ and $ k $ is odd; conversely, if $n$ is even, $k$ is even, and it can also be equal to $0$).
		
		Recall that our aim is to find a $n\times n$ X-matrix whose characteristic polynomial is $f$.
		A generic X-matrix is $A=(a_{i,j})$ where $a_{i,j}=0$ whenever $i\neq j$ and $i+j \neq n+1$. We also recall (Remark \ref{Char:X}) that the characteristic polynomial of an X-matrix is the product of the characteristic polynomials of the matrices 
		\[ \begin{bmatrix}
		a_{i,i} & a_{i,n-i+1}\\
		a_{n-i+1,i} & a_{n-i+1,n-i+1}\\
		\end{bmatrix} , \] multiplied by $(t-a_{n+1/2,n+1/2})$ if $n$ is odd. The next task is, then, to construct the elements $a_{i,j}$ to make $f(t)$ equal to the characteristic polynomial of $A$.
		
		Suppose $n$ is even: indeed, if $n$ is odd, we can set $a_{n+1/2,n+1/2}=\alpha_{k+1/2}$ (it exists because we remarked that if $n$ is odd, $k>0$ and it's an odd number), and go on considering the $(n-1)\times (n-1)$ matrix obtained ignoring the $(n+1)/2$-th row and column. In this case, $k$ is even as well. We start building our X-matrix from its central element; place the $k\times k$ diagonal matrix $\Delta_k=\diag(\alpha_1,\dots,\alpha_n)$ in the middle of our $n\times n$ matrix. This means that for all $j=1,\dots,k/2$ we set
		\[
		\begin{bmatrix}
		a_{l+j,l+j} & a_{l+j,n-l-j+1} \\
		a_{n-l-j+1,l+j} & a_{n-l-j+1,n-l-j+1} \\
		\end{bmatrix}
		=
		\begin{bmatrix}
		\alpha_j & 0 \\
		0 & \alpha_{k+1-j}\\
		\end{bmatrix}.
		\]
		The characteristic polynomial of this matrix is
		$(t-\alpha_{j})\cdot (t-\alpha_{k+1-j})$. The only blocks remaining are the external ones. Then, for all $j=k+1,\dots,l$, we can consider the factor $t^2+\beta_j t+\gamma_j$ and its $2\times 2$ companion matrix, that is
		\[
		C_j=
		\begin{bmatrix}
		0 & -\beta_j \\
		1 &  -\gamma_j \\
		\end{bmatrix},
		\]
		whose characteristic polynomial is $t^2+\beta_j t+ \gamma_j$, and we can set
		\[
		\begin{bmatrix}
		a_{j-k,j-k} & a_{j-k,n-j+k+1}\\
		a_{n-j+k+1,j-k} & a_{n-j+k+1,n-j+k+1}\\
		\end{bmatrix} = C_j.
		\]
		In conclusion, we obtain a matrix of this shape:
		\[C=
		\begin{bmatrix}
		0 &\dots & & & & & & & \dots & -\beta_{k+1} \\
		\vdots& \ddots &&&&&&& \udots & \vdots \\
		&& 0 & \dots &&&\dots & -\beta_{l} && \\
		&&\vdots & \alpha_1 &&& 0 &\vdots && \\
		&&&& \ddots & \udots &&&& \\
		&&&& \udots & \ddots &&&& \\
		&&\vdots & 0 &&& \alpha_k &\vdots && \\
		&& 1 &\dots &&&\dots & -\gamma_{l} && \\
		\vdots & \udots &&&&&&& \ddots &\vdots \\
		1 &\dots &&&&&&&\dots & -\gamma_{k+1} \\
		\end{bmatrix}.
		\]
		This is the matrix we were looking for, because $ C $ is an $ n \times n$ matrix such that
		\[
		\det(tI-A)=\prod_{i=1}^r (t-\alpha_i) \cdot \prod_{i=k+1}^l \det(tI-C_i)= \prod_{i=1}^r (t-\alpha_i) \cdot \prod_{i=k+1}^l (t^2+\beta_i t+ \gamma_i)=f(t).
		\]
		Note that $A$ is a $n\times n$ X-matrix whose characteristic polynomial is $f$.
	\end{proof}
	
	\begin{remark}
		Alternatively, one may provide a companion matrix of the form
		\[C=
		\begin{bmatrix}
		0 &\dots & & & & & & & \dots & -\beta_{k+1} \\
		\vdots& \ddots &&&&&&& \udots & \vdots \\
		&& 0 & \dots &&&\dots & -\beta_{l} && \\
		&&\vdots & 0 &&& -\alpha_1 \alpha_k &\vdots && \\
		&&&& \ddots & \udots &&&& \\
		&&&& \udots & \ddots &&&& \\
		&&\vdots & 1 &&& \alpha_1 + \alpha_k &\vdots && \\
		&& 1 &\dots &&&\dots & -\gamma_{l} && \\
		\vdots & \udots &&&&&&& \ddots &\vdots \\
		1 &\dots &&&&&&&\dots & -\gamma_{k+1} \\
		\end{bmatrix}.
		\]
		which is composed of two constant blocks on the left: a zero block and an anti-identity block.
	\end{remark}
	
	\noindent The above result shows that for every monic polynomial in $ \mathbb{R} $ there exists an X-companion matrix. Still considering the case $\K=\R$, we note that a bi-symmetric X-matrix is in particular symmetric, and therefore diagonalizable. As a consequence, its characteristic polynomial factors into linear terms, so that we cannot cover the whole set of monic polynomials over $\R$ using bi-symmetric X-matrices only.

\begin{figure}[b!]
	\includegraphics[width=1\textwidth]{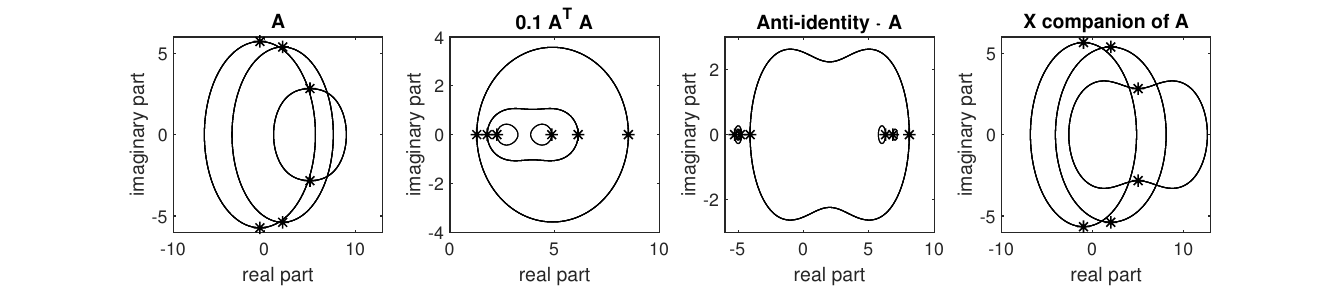}
	\caption{Eigenvalue Inclusion via Cassini ovals for the X matrix $A$, its symmetric product $A^T A$, multiplied by the anti-diagonal, and in X companion form.}
	\label{fig:cassini}
\end{figure}

\subsection{Eigenvalue Inclusion for X-Matrices}
Gerschgorin's Disk Theorem and the expansions to Cassini's ovals \cite{Varg04} 
allow for locating eigenvalues 
within geometrical regions (the analytical expressions of eigenvalues and eigenvectors of X-matrices are provided in the Appendix). 
For X-matrices given by $A=(a_{ij})=\diag(d_i)+\diag(e_i)J$  the disks are centered at the diagonal entries $d_i$ and have radius $|e_i|$ 
when considering row sums, 
(or $|e_{n-i+1}|$ when considering column sums which is not persued here further). 
The union of this disks includes the spectrum. Cassini's ovals are given by $C(z_1,z_2,d)=\{ s \in\mathbb{C}: |s-z_1|\,|s-z_2| \leq d \}$. 
From Brualdi's results, only elements connected by a cycle need to be considered for covering the spectrum. 
For X-matrices, rows $i$ and $n-i+1$ form nontrivial cycles.
For odd $n$, the matrix center element $a_{{\frac{n+1}{2}},\frac{n+1}{2}}$ is already an eigenvalue.
 Hence, the spectrum of the X-matrix $A$ satisfies
$\sigma(A)\subset C(d_i,d_{n-i+1}, |e_i e_{n-i+1} |)$. Therefore $n$ Gerschgorin disks are replaced by 
$\frac{n}{2}$ ovals. For $n$ odd, we have $\frac{n-1}{2}$ ovals together with the point of the matrix center. 

\begin{theorem}
Eigenvalues of real X matrices are located on the boundary of the Cassini ovals.
\end{theorem}
\begin{proof}
The characteristic polynomial decomposes into a product of quadratic polynomial and the possible linear part from the center element, 
if $n$ is odd. These quadratic polynomials are obtained from
\begin{equation*}
\det\left(\begin{bmatrix}t-d_i & -e_i\\ -e_{n-i+1} & t-d_{n-i+1}\end{bmatrix} \right)=(t-d_i)(t-d_{n-i+1}) - e_i e_{n-i+1}, 
\end{equation*}
the eigenvalues $\lambda_{i,1}$, $\lambda_{i,2}$ satisfy $(\lambda_i-d_i)(\lambda_i-d_{n-i+1}) = e_i e_{n-i+1}$,
so that 
\[
\lambda_{i,1}, \lambda_{i,2} \in \partial C(d_i,d_{n-i+1}, |e_i e_{n-i+1} |),\text{    } i\leq\left\lfloor\frac{n}{2}\right\rfloor.\qedhere
\]
\end{proof}
Figure \ref{fig:cassini} shows the Cassini curves for a specific matrix, namely $A=\diag(3, 3,-2,1,1,7) +\diag(2-,-5,-5,7,6,6)J$ and some derived matrices. 
Note that the shapes vary greatly between eights and ovals.
Moreover, the eigenvalues (denoted by asterisks in the graphs) are directly located on the the curves. 


\begin{remark}
Letting $D=\diag(d_i)$ and $E=\diag(e_i)$
consider $f_{n-i+1}=-d_id_{n-i+1}+e_ie_{n-i+1}$, $f_i=1$ for $i<\frac{n}{2}$, 
$g_{n-i+1}=d_i+d_{n-i+1}$, $g_i=0$ for $i<\frac{n}{2}$.
Then $Y=\diag(g)+\diag(f)J$ has the same spectrum as $X$.
Consider the matrices \begin{equation*}
X=\begin{bmatrix}  6& & &6\\ & -5 & -8 & \\ & -4 &6 & \\-5 & & &1
\end{bmatrix}, \quad C=\begin{bmatrix}  & & & 1\\ &  & 1 & \\ & 62 &1 & \\-36 & & &7\end{bmatrix}.
\end{equation*}
They have the same spectrum, but different Cassini ovals.
\end{remark} 

\section{Conclusions}
This work has investigated the family of X-matrices that shares interesting properties: any $ A \in \X $ is invariant under transposition, the sum of two elements in $ \X $ is in $ \X $ and so is their product. Thus, any matrix function obtained via a convergent Taylor series expansion maps an element of $ \X $ into an element of $ \X $. 
Also, if $ A \in \X $ is invertible, its inverse is in $ \X $. We have shown that the set of bi-simmetric  X-matrices ($ \X^b $) is a commutative subring of $\Mat(n,\K)$. Moreover, when the inverse of an element in $ \X^b $  exists, it is in $ \X^b $.

We have then seen that for any monic polynomial with coefficients in $ \mathbb{R} $, 
there exists an element of $\X  $ whose characteristic polynomial is the monic polynomial of interest. 
This is not true for diagonal matrices in $ \mathbb{R} $. Bi-symmetric X-matrices are also not rich enough 
to cover every polynomial with coefficients in $ \R $.

Regarding eigenvalue inclusion, for elements of $ \X $, eigenvalues lie on the boundary of Cassini's ovals. 

Finally, we note that shape invariance has implications in numerical linear algebra, and several families of matrices which are shape invariant under multiplication have been studied. The present family is not only shape invariant under multiplication, but under transposition as well.

\section*{Acknowledgements}
The authors wish to thank Professors Carlo Baldassi and Richard Brualdi for useful comments.

\section{Appendix: Eigenvalues, Eigenvectors of X-matrices}
The characteristic polynomial of an X-matrix is made of products of quadratic polynomials (times the matrix center root $t-a_{\frac{n+1}{2},\frac{n+1}{2}}$ if $n$ is odd). 
When equality to zero is of interest, we can then solve separately each of the quadratic equations obtaining
the system of equations, 
\begin{equation*}
\delta_i - \theta_i t +t^2=0, \qquad i=1,2,\dots,\lfloor \frac{n}{2}\rfloor
\end{equation*}
where $\delta_i=a_{i,i}a_{n-i+1,n-i+1} -a_{i,n-i+1}a_{n-i+1,i}$ and $\theta_i=a_{i,i}+a_{n-i+1,n-i+1}$ are determinant and trace, respectively, of the $2\times 2$
submatrix $\left(\begin{smallmatrix}a_{i,i} & a_{i,n-i+1}\\ a_{n-i+1,i} & a_{n-i+1,n-i+1}\end{smallmatrix}\right)$.

When $n$ is odd we have the eigenvalue $\lambda_{(n+1)/2}=a_{(n+1)/2,(n+1)/2}$.

The remaining eigenvalues are found from Vieta's formulas $\lambda_{i,1} + \lambda_{i,2} = \theta_i$, $\lambda_{i,1} \cdot \lambda_{i,2} = \delta_i$, so that
$\lambda_{i,1}= \frac{\theta_i}{2}+\frac{1}{2}\sqrt{4\delta_i-\theta_i^2}$, $\lambda_{i,2}= \frac{\theta_i}{2}-\frac{1}{2}\sqrt{4\delta_i-\theta_i^2}$.

Now, for
bi-symmetric matrices, we have $a_{i,i}=a_{n-i+1,n-i+1}$ and $%
a_{i,n-i+1}=a_{n-i+1,i}$, so that $\theta_i=2a_{i,i}$ and $\delta_i=a_{i,i}^2-a_{i,n-i+1}^2=\left(a_{i,i}-a_{i,n-i+1}\right)\left(a_{i,i}+a_{i,n-i+1}\right)$.
Thus, for bi-symmetric X-matrices, eigenvalues have a particularly simple
form, namely $\lambda_{i,1}=a_{i,i}-a_{i,n-i+1}$, $\lambda_{i,2}=a_{i,i}+a_{i,n-i+1}$,
Let us now have a look at the eigenvectors of X matrices. They are of the form $\alpha_i e_i+\beta_i e_{n-i+1}$. 
For $\lambda_{i,1}$, setting $\alpha_i= a_{i,n-i+1}$ one obtains
$\beta_i= \frac{a_{n-i+1,n-i+1}-a_{i,i}}{2}+\frac{1}{2}\sqrt{4\delta_i-\theta^2_i}$. For $\lambda_{i,2}$, setting $\beta_i= a_{n-i+1,i}$ one obtains
$\alpha_i= \frac{a_{i,i}-a_{n-i+1,n-i+1}}{2}-\frac{1}{2}\sqrt{4\delta_i-\theta^2_i}$.

Thus, for bi-symmetric matrices this calculation simplifies to the pairs $(\alpha_i,\beta_i)=(1,1)$ and $(\alpha_i,\beta_i)=(1,-1)$.\\


\end{document}